\documentclass[11pt]{article}
\usepackage{amsmath,amssymb}

\newtheorem{propo}{{\bf Proposition}}[section]
\newtheorem{coro}[propo]{{\bf Corollary}}
\newtheorem{lemma}[propo]{{\bf Lemma}} \newtheorem{theor}[propo]{{\bf
Theorem}} 

\newenvironment{proof}{{\bf Proof.}}{$\Box$}

\def\N{{\mathbb N}}
\begin{document}
\vspace*{1.0in}

\begin{center} ZINBIEL ALGEBRAS ARE NILPOTENT

\end{center}
\bigskip

\centerline {David A. Towers} \centerline {Department of
Mathematics, Lancaster University} \centerline {Lancaster LA1 4YF,
England}
\bigskip

\begin{abstract}
In this paper we show that every finite-dimensional Zinbiel algebra over an arbitrary field is nilpotent, extending a previous result by other authors that they are solvable.
\end{abstract}


\section{Introduction}
Zinbiel algebras were introduced by J.-L. Loday \cite{Loday95} in 1995. They are the Koszul dual of Leibniz algebras and J.M. Lemaire (see \cite{Lod01}) proposed the name of Zinbiel which is obtained by writing Leibniz backwards. Leibniz algebras were defined by Loday in 1993 (see \cite{Loday93}). They are a particular case of non-associative algebras and a non-anticommutative generalization of Lie algebras. In fact, they inherit an important property of Lie algebras: the right-multiplication operator is a derivation. Many well-known results on Lie algebras can be extended to Leibniz algebras. In some papers, like \cite{LP,O2005}, the authors study the cohomological and structural properties of Leibniz algebras. Ginzburg and Kapranov introduced and analysed the concept of Koszul dual operads \cite{GK}. Starting from this concept, it was proved in \cite{Loday95} that the dual of the category of Leibniz algebras is defined by the category determined by the so-called Zinbiel identity:
\[  [[x,y],z]=[x,[y,z]]+[x,[z,y]]
\]
Some properties of Zinbiel algebras were studied in \cite{AOK,D,DT}. Filiform Zinbiel algebras were described and classified in \cite{AOK,CKKK,CCGO}. The classification of complex Zinbiel algebras up to dimension $4$ was obtained in \cite{DT} and \cite{O2002}. Finally, a partial classification of the $5$-dimensional case was done in \cite{ACK}.
\par

More generally, in \cite{DT} the authors proved that every finite-dimensional Zinbiel algebra over an algebraically closed field is solvable and it is nilpotent over the complex number field. The requirement of the algebraically closed field can be removed, of course, since, if $\Omega$ is an extension field of $F$, $Z\otimes_F\Omega$ is solvable (respectively nilpotent) if and only if $Z$ is. Their results, therefore, show that all Zinbiel algebras are solvable, and that, over a field of characteristic zero, they are nilpotent. The purpose of this paper is to show that the restriction to characteristic zero in this latter result is unnecessary, and to provide an easier proof.
\par

Throughout, $Z$ will denote a finite-dimensional Zinbiel algebra over an arbitrary field $F$. We define the following series:
\[ Z^1=Z,Z^{k+1}=[Z,Z^k] \hbox{ and } Z^{(1)}=Z,Z^{(k+1)}=[Z^{(k)},Z^{(k)}] \hbox{ for all } k=2,3, \ldots
\]
Then we define $Z$ to be {\em nilpotent} (resp. {\em solvable}) if $Z^n=0$ (resp. $ Z^{(n)}=0$) for some $n \in \N$. In a nilpotent Zinbiel algebra, every product of $n$ elements is zero. An ideal $A$ of $Z$ is said to a {\em zero} ideal if $A^2=0$. The {\em Frattini subalgebra}, $F(Z)$, of $Z$ is the intersection of the maximal subalgebras of $Z$, and the {\em Frattini ideal}, $\phi(Z)$, is the largest ideal contained in $F(L)$.

\section{Main results}
\begin{lemma}\label{1} Let $B$ be a right ideal of a Zinbiel algebra $Z$. Then $[Z,B]$ is an ideal of $Z$.
\end{lemma}
\begin{proof} We have $[Z,[Z,B]]\subseteq [Z^2,B]+[Z,[B,Z]]\subseteq [Z,B]$, and $[[Z,B],Z]\subseteq [Z^2,B]\subseteq [Z,B]$.
\end{proof}

\begin{propo}\label{min} Let $A$ be a minimal ideal of a Zinbiel algebra $Z$ and let $B$ be a minimal right ideal of $Z$ with $B\subseteq A$. Then $A=B$.
\end{propo}
\begin{proof}  Clearly, $[Z,B]\subseteq [Z,A]\subseteq A$, so $[Z,B]=0$ or $[Z,B]=A$, by Lemma \ref{1}. The former implies that $B$ is an ideal of $Z$, and so $B=A$.
\par

So suppose that $$[Z,B]=A.$$ Now $[B,Z^2]$ is a right ideal inside $B$. If $B\subseteq [B,Z^{(k)}]$, then $$B\subseteq [[B,Z^{(k)}],Z^{(k)}]\subseteq [B,Z^{(k+1)}]$$ for all $k\geq 2$. Since $Z$ is solvable, this implies that $$[B,Z^2]=0.$$
If $[B,Z]=B$, then $B=[[B,Z],Z]\subseteq [B,Z^2]=0$, so $$[B,Z]=0.$$
Now, 
\begin{align*}
[Z,[Z^2,B]]&\subseteq [Z^3,B]+[Z,[B,Z^2]]\subseteq [Z^2,B] \hbox{ and  } \\
[[Z^2,B],Z]&\subseteq [[Z^2,Z],B]\subseteq [Z^2,B],
\end{align*} so $[Z^2,B]$ is an ideal of $Z$. As it is inside $A$ we have that $[Z^2,B]=A$ or $[Z^2,B]=0$. 
\par

The former implies that $B\subseteq [Z^2,B]$. But $B\subseteq [Z^{(k)},B]$ implies that $B\subseteq [Z^{(k)},[Z^{(k)},B]]\subseteq [Z^{(k+1)},B]$ for all $k\geq 2$, since $[Z^{(k)},[B,Z^{(k)}]]=0$. As $Z$ is solvable we have that $$[Z^2,B]=0.$$
But now $[A,Z]=[[Z,B],Z]\subseteq [Z^2,B]=0$ and $[Z,A]=[Z,[Z,B]]\subseteq [Z^2,B]+[Z,[B,Z]]=0$. It follows that $\dim A=1$ and $B=A$.
\end{proof}

\begin{coro}\label{central} If $A$ is a minimal ideal of the Zinbiel algebra $Z$, $[Z,A]=[A,Z]=0$ and $\dim A=1$.
\end{coro}
\begin{proof} It is clear that $[A,Z]$ is a right ideal inside $A$. If $B$ is a minimal right ideal of $Z$ inside $[A,Z]$ we have that $B=[A,Z]=A$, by Lemma \ref{1}. But then $A=[A,Z]=[[A,Z],Z]\subseteq [A,Z^2]$. As before, if $A\subseteq [A,Z^{(k)}]$, then $A\subseteq [[A,Z^{(k)}],Z^{(k)}]\subseteq [A,Z^{(k+1)}]$. Since $Z$ is solvable we have that $[A,Z]=0$.
\par

Now $[Z,A]$ is also a right ideal inside $A$. If $[Z,A]=A$ then $A=[Z,[Z,A]]\subseteq [Z^2,A]$ and a similar argument shows that $[Z,A]=0$.
\end{proof}

\begin{theor} Zinbiel algebras $Z$ over any field are nilpotent.
\end{theor}
\begin{proof} This is a straightforward induction. The result clearly holds if $\dim Z=1$. Suppose it holds for Zinbiel algebras of dimension $<k$ ($k>1$), and suppose that $\dim Z=k$. Let $A$ be a minimal ideal of $Z$. Then $L/A$ is nilpotent, by the inductive hypothesis, so there exists $n$ such that $Z^n\subseteq A$. But now $Z^{n+1}\subseteq [Z,A]=0$, by Corollary \ref{central}.
\end{proof}

\begin{coro}\label{max} Maximal subalgebras of Zinbiel algebras are ideals.
\end{coro}
\begin{proof} Let M be a maximal subalgebra of the Zinbiel algebra $Z$. Then there exists $k$ such that $Z^k\not \subseteq M$ but $Z^{k+1}\subseteq Z$. Since $Z^k$ is an ideal of $Z$, $Z=M+Z^k$. Hence $[Z,M]\subseteq M+Z^{k+1}=M$ and $[M,Z]\subseteq M+[Z^k,Z]$. But an easy induction proof shows that $[Z^k,Z]\subseteq Z^{k+1}\subseteq M$.
\end{proof}

\begin{coro} For every Zinbiel algebra $Z$, $F(Z)=\phi(Z)=Z^2$.
\end{coro}
\begin{proof} By Corollary \ref{max}, for every maximal subalgebra $M$, $Z/M$ is a zero algebra, so $Z^2\subseteq M$ and $Z^2\subseteq F(Z)$. But, clearly, $F(Z)\subseteq Z^2$.
\end{proof}

\end{document}